\documentclass[12pt, reqno]{amsart}

\usepackage{fullpage}
\usepackage{enumerate}
\usepackage{setspace}
\usepackage{listings}
\usepackage{bm}
\usepackage{amsmath,amsfonts,amssymb,graphics,amsthm,comment}
\usepackage{hyperref}
\DeclareMathOperator{\Pic}{Pic}
\DeclareMathOperator{\Supp}{Supp}
\hypersetup{
    colorlinks=true,
    linkcolor=blue,
    citecolor=red,
    urlcolor=blue,
    pdfborder={0 0 0}
} 
\usepackage{dsfont}
\newtheoremstyle{nodefstyle} 
  {}        
  {}        
  {\normalfont} 
  {}        
  {}        
  {}        
  {0pt}     
  {}        

\usepackage[framemethod=TikZ]{mdframed}
\usepackage{lipsum}
\usepackage{tcolorbox}
\mdfdefinestyle{MyFrame}{%
    linecolor=blue,
    outerlinewidth=2pt,
    roundcorner=20pt,
    innertopmargin=\baselineskip,
    innerbottommargin=\baselineskip,
    innerrightmargin=20pt,
    innerleftmargin=20pt,
    backgroundcolor=gray!50!white}
\usepackage{shadethm} 

\usepackage{graphicx}
\usepackage{tikz}
\usetikzlibrary{cd}
\usepackage{xcolor}
\usepackage{wrapfig} 
\usepackage{xcolor}

\usepackage[font=sf, labelfont={sf,bf}, margin=1cm]{caption}

\usepackage{cleveref}
  \Crefname{theorem}{Theorem}{Theorems}
  \Crefname{thm}{Theorem}{Theorems}
  \Crefname{lemma}{Lemma}{Lemmas}

  \Crefname{lem}{Lemma}{Lemmas}
  \Crefname{remark}{Remark}{Remarks}
  \Crefname{prop}{Proposition}{Propositions}
  \Crefname{proposition}{Proposition}{Propositions}
  \Crefname{problem}{Problem}{Problems}
\Crefname{notation}{Notation}{Notations}
\Crefname{shadethm}{Theorem}{Theorems}
\Crefname{claim}{Claim}{Claims}
  \Crefname{defn}{Definition}{Definitions}
  \Crefname{corollary}{Corollary}{Corollaries}
  \Crefname{shadedefn}{Definition}{Definitions}
  \Crefname{section}{Section}{Sections}
  \Crefname{figure}{Figure}{Figures}
  \Crefname{shadeprop}{Proposition}{Propositions}
  \Crefname{exercise}{Exercise}{Exercises}
    \Crefname{assumption}{Assumption}{Assumptions}

\newtheorem{thm}{Theorem}[section]
\newshadetheorem{shadethm}{Theorem}[section]
\definecolor{shadethmcolor}{HTML}{BEFF33}
\definecolor{shaderulecolor}{HTML}{45CFFF}
\newshadetheorem{shadedefn}{Definition}[section]
\definecolor{shadethmcolor}{HTML}{CBFAFA}
\definecolor{shaderulecolor}{HTML}{CBFAFA}
\newshadetheorem{shadeprop}{Proposition}[section]
\definecolor{shadethmcolor}{HTML}{CAFAF2 }
\definecolor{shaderulecolor}{HTML}{CBFAFA}
\newtheorem{lemma}[thm]{Lemma}

\newtheorem{corollary}[thm]{Corollary}

\newtheorem{question}[thm]{Question}

\numberwithin{equation}{section}

\theoremstyle{definition}
\newtheorem{remark}[thm]{Remark}
\newtheorem{example}[thm]{Example}

\theoremstyle{nodefstyle}
\newtheorem*{define}{}

\makeatletter
\def\munderbar#1{\underline{\sbox\tw@{$#1$}\dp\tw@\z@\box\tw@}}
\makeatother

\usepackage{wasysym}

\newcommand{\Jac}{\operatorname{Jac}}
\newcommand{\tors}{\operatorname{tors}}

\def\N{\mathbb{N}}

\def  \p- {p\textunderscore}

\usepackage{extarrows}

\begin{document}

\title[Generalized fruit Diophantine equation and super elliptic curves]{Generalized fruit Diophantine equation and super elliptic curves}
\author{Kalyan Banerjee}
\address{Department of Mathematics, SRM University AP, Mangalagiri-Mandal, Amaravati-522240, Andhra Pradesh, India}
\email{\url{kalyan.ba@srmap.edu.in}}

\author{Kalyan Chakraborty}
\address{Department of Mathematics, SRM University AP, Mangalagiri-Mandal, Amaravati-522240, Andhra Pradesh, India}
\email{\url{kalyan.c@srmap.edu.in}}

\author{Ankita Das}
\address{Department of Mathematics, SRM University AP, Mangalagiri-Mandal, Amaravati-522240, Andhra Pradesh, India}
\email{\url{ankita_das@srmap.edu.in}}

\maketitle
\begin{abstract} In this article, we are interested in  finding rational points on certain superelliptic curves.

\end{abstract}

\section{Introduction}
The Diophantine equations, studied in this paper, belongs to a family that has attracted
considerable attention in recent years under the name of \emph{fruit Diophantine equations}.
This terminology was introduced in the work of
\cite{MajumdarSuryFruit}, where Majumdar and Sury initiated a systematic study
of Diophantine equations of mixed degree exhibiting strong arithmetic rigidity.
Subsequently, Sharma and Vaishya \cite{SharmaVaishyaFruit} investigated a broader class of
such equations, establishing nonexistence results for integer solutions under various
congruence and coprimality conditions.

More recently, Chakraborty and Prakash \cite{ChakrabortyPrakashFruit} introduced a
geometric perspective by associating certain generalized fruit Diophantine equations to
hyperelliptic curves and analyzing their arithmetic via the geometry of the corresponding
Jacobians. In their work, the Diophantine problem is treated for a specific subfamily of
equations for which the associated curves admit particularly favorable structural
properties.

The present work is motivated by the desire to extend this line of investigation to a
\emph{more general family} of fruit Diophantine equations. By allowing greater flexibility
in the defining parameters, we study a wider class of equations while retaining sufficient
arithmetic control to analyze their integer solutions.
\subsection{Outline of the paper.}In \Cref{sec:diophantine} we prove that a particular family of Diophantine equations \eqref{eq1} has no integer solution. First, we assume that \eqref{eq1} admits an integer solution $(x,y,z)$. We consider the possible cases that may arise and reach at our conclusion by the method of contradiction. The main ingredient of this proof is the \emph{modular arguments}. \Cref{th1} helps us to derive \Cref{th6} in \Cref{sec:Superelliptic curves} which is the ultimate goal of this paper.

In order to translate the Diophantine rigidity obtained in \Cref{sec:diophantine} into geometric
statements about algebraic curves, we require several tools from arithmetic geometry. In \Cref{sec:term} we discuss some basic geometric and arithmetic notions and results.

In \Cref{sec:Superelliptic curves} we use \emph{the method of good reduction} for the family of superelliptic curves \eqref{eq:final} to show that no nonzero torsion can be reduced to zero. Under some suitable coprimality conditions on the order of the reduced Jacobians, we yield the vanishing result for torsion in \Cref{th6}. The proof of \Cref{th6} proceeds through a sequence of reductions.
First, Lemma\autoref{lem:1} establishes a discriminant divisibility condition satisfied by torsion points on the superelliptic curve.
This condition implies, in Lemma\autoref{lem:2}, that the $y$--coordinates of such points belong to a finite, explicitly determined set.
Lemma\autoref{lem:3} then shows that every torsion divisor class in the Jacobian is supported on rational points whose $y$--coordinates lie in this finite set.
To control such divisor classes, Lemma\autoref{lem:4} relates rational torsion to torsion in the Jacobian over finite fields via reduction at primes of good reduction.
Finally, Lemma\autoref{lem:5} shows that under suitable coprimality conditions on the orders of the reduced Jacobians, the relevant subgroup contains no nontrivial torsion, which completes the proof of \Cref{th6}. This result leads us to show that \eqref{eq:final} has torsion-free Mordell-Weil group over $\mathbb Q$ which in turn proves that \eqref{eq:final} has no non trivial integer solution.

In \Cref{sec:ex} we have illustrated several examples to corroborate the main result in \Cref{th6} and in \Cref{sec:open} we discuss some open questions.

\subsection{Summary of notations}\label{sec:notation}

Throughout the paper, we use the following notations.

\begin{itemize}

  \item \textbf{$\mathbb{Z}$:}
  The ring of integers.
  
  \item \textbf{$\mathbb{Q}$:}
  The field of rational numbers.
  
  \item \textbf{$\mathbb{F}_p$:}
  The finite field with $p$
  elements.

  \item \textbf{$a \equiv b \pmod{n}$:}
  Congruence modulo $n$, meaning that $n$ divides $a-b$.

  \item \textbf{$\gcd(a,b)$:}
  The greatest common divisor of the integers $a$ and $b$.

  \item \textbf{$f(X)$:}
  A separable polynomial of odd degree $d_1 \ge 3$ in $\mathbb{Z}[X]$ unless otherwise stated.

  \item \textbf{$\Delta(f)$:}
  The discriminant of the polynomial $f$.

  \item \textbf{$\Delta(f)(y)$:}
  The discriminant of $f$, regarded as an explicit integer-valued function of the integer
  parameter $y$ when the coefficients of $f$ depend on $y$.

  \item \textbf{$\infty$:}
  The unique point at infinity on the smooth projective model of the curve under
  consideration.

  \item \textbf{$\Pic(C)$:}
  The Picard group of a smooth projective curve $C$.

  \item \textbf{$\Pic^0(C)$:}
  The subgroup of $\Pic(C)$ consisting of divisor classes of degree zero.

  \item \textbf{$\Jac(C)$:}
  The Jacobian variety of a smooth projective curve $C$.

  \item \textbf{$[P-\infty] \in \Jac(C)(\mathbb{Q})$:}
  The divisor class associated to a rational point $P \in C(\mathbb{Q})$ via the
  Abel--Jacobi embedding.

  \item \textbf{$\Jac(C)(\mathbb{Q})_{\mathrm{tors}}$:}
  The torsion subgroup of the Jacobian of $C$ over $\mathbb{Q}$.

  \item \textbf{$\Supp(D)$:}
  The support of a divisor $D = \sum_P n_P P$, defined as
  $\Supp(D) = \{ P : n_P \neq 0 \}$.

  \item \textbf{$v_p(\cdot)$:}
  The $p$-adic valuation on $\mathbb{Q}$, normalized so that $v_p(p)=1$.


  \item \textbf{$\rho_p : \Jac(C)(\mathbb{Q}) \to \Jac(C)(\mathbb{F}_p)$:}
  The reduction map induced by the N\'eron model of the Jacobian at a prime of good
  reduction.

  \item \textbf{$\#\Jac(C)(\mathbb{F}_p)$:}
  The cardinality of the finite abelian group $\Jac(C)(\mathbb{F}_p)$.
\end{itemize}

\subsection{Acknowledgements.} The authors thank SRM AP for hosting this project.

\section{Diophantine equations have no integer solution}\label{sec:diophantine}
The purpose of this section is to establish the nonexistence of integer solutions for a
specific family of Diophantine equations under explicit arithmetic constraints on the
parameters. The argument proceeds by fixing one variable and analyzing the resulting
equation through congruence obstructions and descent-type considerations. This
Diophantine analysis provides the arithmetic foundation for the geometric constructions
introduced later.

We begin by treating the general case and proving a complete nonexistence result for
integer solutions under the given congruence and coprimality assumptions.
\begin{thm}
The equation 
\[ax^{d_1}-y^{d_2}-z^2+xyz-b=0\] has no integer solution for fixed a and b such that,
\[a\equiv 1 \pmod {12}\]
\[b=2^{d_1}a-3\]
\[d_1 \geq 3 , 3\mid {d_1}\]
\[d_2 \geq 2, 2\mid {d_2}\]
\[gcd(d_1, d_2)=1.\]\label{th1}
\end{thm}
The proof proceeds by fixing the variable $x$ and distinguishing cases according to
its parity. In each case, suitable congruence arguments lead to contradictions, ruling
out the existence of integer solutions.
\begin{proof}
Consider
\begin{equation}
    ax^{d_1}-y^{d_2}-z^2+xyz-b=0 \label{eq1}
\end{equation}
If possible, let $(x,y,z)$ be an integer solution of \eqref{eq1}. Fix $x=\alpha$. Then \eqref{eq1}  can be rewritten as 
\begin{align*}
 a\alpha^{d_1}-y^{d_2}-z^2+\alpha yz-b=0    
\end{align*}
\begin{equation}
  y^{d_2}+z^2+b=a\alpha^{d_1}+\alpha yz.  \label{eq2} 
\end{equation}
\vspace{2mm}
We consider the cases of $\alpha$ being even or odd separately. \\
\vspace{2mm}
\textbf{Case 1: $\alpha$ is even.}

We can rewrite \eqref{eq2} as
$$y^{2s}+z^2+b=a\alpha^{d_1}+\alpha yz,$$  where $d_2=2s,s\in\N$. Thus we obtain,
\[ (y^s)^2-2.y^s.\frac{\alpha z}{2y^{s-1}}+\left( \frac{\alpha z}{2y^{s-1}} \right)^2+z^2-\left( \frac{\alpha z}{2y^{s-1}} \right)^2=a\alpha ^{d_1}-b\]
\[\left( y^s-\frac{\alpha z}{2y^{s-1}}\right)^2-\left( \frac{\alpha ^2}{4y^{2(s-1)}}-1\right)z^2=a\alpha ^{d_1}-b\]
\begin{equation}
\left( y^{2s-1}-\frac{\alpha z}{2}\right)^2-\left( \frac{\alpha ^2}{4}-y^{2(s-1)}\right)z^2=y^{2(s-1)}[ a\alpha ^{d_1}-b.] \label{eq3}   
\end{equation}
Now let us set,
\begin{equation*}
    Y= y^{2s-1}-\frac{\alpha z}{2},
    \end{equation*}
    \begin{equation*}
        \beta= \frac{\alpha}{2},
    \end{equation*}
    \begin{equation*}
        k=y^{s-1},
    \end{equation*}
    \begin{equation*}
        Z=z.
    \end{equation*}
    Then, \eqref{eq3} becomes,
\begin{equation}    
Y^2-\left(\beta ^2-k^2\right )Z^2=k^2\left[a\alpha ^{d_1}-b\right] =k^2\left[a2^{d_1}\beta ^{d_1}-a2^{d_1}+3\right]. \label{eq4}
\end{equation}

    If $\beta$ is even, then reducing \eqref{eq4} modulo 4, we get,
    \begin{equation*}
        Y^2 \equiv 3 \pmod{4} 
    \end{equation*} or 
    \begin{equation}
        Y^2+Z^2 \equiv 3 \pmod{4}
    \end{equation}
    which is not possible in $\mathbb{Z}/{4\mathbb{Z}}$.\\
    If $\beta$ is odd, then reducing \eqref{eq4} modulo 4, we get,
    \[Y^2 \equiv 3 \pmod{4}\] or
    \begin{equation}
        Y^2-Z^2 \equiv 3 \pmod{4}
    \end{equation}
    which is agin not possible in $\mathbb{Z}/{4\mathbb{Z}}$.
    \vspace{1mm}

    \textbf{Case $2$: $\alpha$ is odd.}
    \vspace{1mm}
    
    First, we take $\alpha =2n+1$, for some $n \in \mathbb{Z}$. Then we can rewrite \eqref{eq2} as,
    \[y^{2s}+z^2+b=a\alpha ^{d_1}+\alpha yz\]
    \[ y^{2s}+z^2+a2^{d_1}-3=a(2n+1)^{d_1}+(2n+1)yz\]
    \[  y^{2s}+z^2-(2n+1)yz=a(2n+1)^{d_1}-a2^{d_1}+3.\]
    Since $y^{2s}+z^2+yz \equiv 0 \pmod{2}$ has the only solution $y \equiv z \equiv 0$ in $\mathbb{Z}/{2\mathbb{Z}}$, that is, $y$ and $z$ both are even, \eqref{eq2} becomes,
    \[a\alpha ^{d_1}-b \equiv 0 \pmod{4}.\]
    If we write $a=12l+1$ for some $l \in \mathbb{Z}$, then the last equation will be reduced to,
    \[(12l+1)\alpha ^{d_1}-(12l+1)2^{d_1}+3 \equiv 0 \pmod{4}\]
    \[ \alpha ^{d_1}-2^{d_1}+3 \equiv 0 \pmod{4}\]
    \[  \alpha ^{d_1} \equiv 1\pmod{4}\]
    \[ \alpha \equiv 1 \pmod{4}.\]
    Now we consider,
    \begin{align}
    \left(y^{2s-1}-\alpha \frac{z}{2}\right)^2-\left(\alpha ^2-4y^{2(s-1)}\right)\frac{z^2}{4}=y^{2(s-1)}\left[a\alpha ^{d_1}-b\right]\\
    \left(y^{2s-1}-\alpha \frac{z}{2}\right)^2-\left(\alpha ^2-4y^{2(s-1)}\right)\left(\frac{z}{2}\right)^2=y^{2(s-1)}\left[a\alpha ^{d_1}-b\right].\label{odd eq}
    \end{align}
    If we set,
        \[Y=y^{2s-1}-\frac{\alpha z}{2},\]
        \[Z=\frac{z}{2},\]
        \[k=y^{s-1},\]
        then \eqref{odd eq} becomes,
        \begin{equation}
            Y^2-\left(\alpha ^2-4k^2\right)Z^2=k^2\left[a\alpha ^{d_1}-b\right], \label{eq7}
        \end{equation}
        where,
        \[\alpha \equiv 1 \pmod{4}\]
        \[a \equiv 1 \pmod{12}\]
        \[b=a2^{d_1}-3.\]
        At this point, three different subcases are needed to be considered separately.\\
        \vspace{2mm}
        \textbf{Subcase 1:} $\alpha \equiv 1\pmod{12}.$\\
        \vspace{2mm}
        Write $\alpha =12l+1$ for some $l \in \mathbb{Z}$. Then $\alpha +2\equiv 0 \pmod{3}.$
        Reducing \eqref{eq7} modulo $3$, we get,
        \[Y^2 \equiv k^2\left[a\alpha ^{d_1}-b\right] \pmod{3}\]
        \[ Y^2 \equiv k^2\alpha ^{d_1}-k^2\left [a2^{d_1}-3\right]\pmod{3}\] \hfill
        \[\equiv k^2\alpha ^{d_1}-k^22^{d_1} \pmod{3}\]\hfill
        \[\equiv k^2\left[1-2^{d_1}\right] \pmod{3}\]\hfill
        \[\equiv -k^2 \pmod{3}.\]
        Now, $k=y^{s-1}$ is even as $y$ is even and $k$ can be either $3p$, $3p+1$ or $3p+2$, for some $p \in \mathbb{Z}$. For $k=3p+1$ to be even, $p$ needs to be odd. Let $p=2s+1$, for some $s \in \mathbb{Z}$. Then $k=3(2s+1)+1=6s+4=2(3s+2)$, which implies,
        \begin{align*}
        k^2&=4(9s^2+2\cdot2\cdot3s+4)&\\
        &\equiv 1 \pmod{3}.&
        \end{align*}
        For $k=3p$ to be even, $p$ needs to be even. Let $p=2s$, for some $s \in \mathbb{Z}$. Then $k=3p=3\cdot2s=6s$, which implies,
        \[k^2 \equiv 36s^2 \equiv 0 \pmod{3}.\]
        For $k=3p+2$ to be even, $p$ needs to be even. Let $p=2s$, for some $s \in \mathbb{Z}$. Then $k=3(2s)+2=2(3s+1)$, which implies,
        \begin{align*}
        k^2&=4(9s^2+2\cdot3s\cdot1+1)&\\
        &\equiv 1 \pmod{3}.&
        \end{align*}
    Here our claim is that $3$ does not divide $Y$. Let us assume, $3$ divides $Y$, which in turn implies that 
    $3$ divides $y^{2s-1}-\frac{\alpha z}{2}$. Thus 
    $3$ divides both $y^{2s-1}$ and $z$, which is same as $3$ divides both $y$ and $z$. We know that $\alpha \equiv 1 \pmod{12}$ which is equivalent to $\alpha \equiv 1 \pmod{3}$. This implies that $\alpha =3q+1$ for some $q \in \mathbb{Z}$.\\
    Since $\alpha$ is odd, $q$ must be even. Let $q=2t,$ for some $ t\in \mathbb{Z}$. Therefore, $\alpha = 6t+1.$ Then,
    \[y^{2s}+z^2+b=a(6t+1)^{d_1}+(6t+1)yz\]
    \[ y^{2s}++z^2-(6t+1)yz=a(6t+1)^{d_1}-a2^{d_1}+3\]
    \[ y^{2s}+z^2-yz \equiv -1\pmod{3} \equiv 2\pmod{3}.\]
    Therefore $y \equiv z \equiv 0 \pmod{3}$ can not be a solution of the above equation. So $3$ does not divide both $y$ and $z$ which in turn implies that $3$ does not divide $Y$. Hence, $k^2 \equiv 1 \pmod{3}$ and $Y^2 \equiv -1 \pmod{3} \equiv 2 \pmod{3}$, which is a contradiction as $2$ is not square modulo $3$. Let us move on to the next subcase.\\
    \vspace{2mm}
    \textbf{Subcase 2:} $\alpha \equiv 9 \pmod{12}$\\
    \vspace{2mm}
    There exists a prime $p$ such that $p \equiv 5 \pmod{12}$ or $p \equiv 7 \pmod{12}$ and $\alpha \equiv 2 \pmod{p}.$
    Then,\[Y^2 \equiv k^2\left[a\alpha ^{d_1}-b\right] \pmod{p}\]
    \[ Y^2 \equiv k^2\cdot a\cdot2^{d_1}-k^2\cdot a\cdot2^{d_1}+k^2\cdot3 \pmod{p}\]
    \[ Y^2 \equiv 3 \pmod{p} \equiv 3 \pmod{p}.\]
    This leads to a contradiction as 3 is not a quadratic residue modulo $p$.\\
    \vspace{2mm}
    \textbf{Subcase 3:} $\alpha \equiv 5 \pmod{12}$\\
    \vspace{2mm}
    Then $\alpha -2 \equiv 0 \pmod{3}$. Let $\alpha = 3t+2$, for some $t \in \mathbb{Z}$. So $Y^2 \equiv k^2\left[a(3t+1)^{d_1}-a2^{d_1}+3\right] \pmod{3} \equiv 0 \pmod{3}.$
    Further substituting $Y=3m$ and $\alpha =12n+5$ for some $m,n \in \mathbb{Z}$ in \eqref{eq7}, we get
    \[9m^2-(12n+3)(12n+7)Z^2=k^2\left [a(12n+5)^{d_1}-a.2^{d_1}+3\right]\]
    \[ 3m^2-(4n+1)(12n+7)Z^2=ak^2(4n+1)\sum _{i=0}^{d_1-1}(12n+5)^{d_1-1-i}.2^i+k^2\]
    \[ -(n+1)Z^2 \equiv k^2 \pmod{3} \equiv 1 \pmod{3}\]
    \[ n \equiv 1 \pmod{3}.\]
    Therefore, $\alpha \equiv 17 \pmod{36}$ and $\alpha -2 \equiv 0 \pmod{3}.$ Hence, there exists a prime $p$ such that $p \equiv 5 \pmod{12}$ or $p \equiv 7 \pmod{12}$ and $p$ divides  $\frac{\alpha -2}{3}$. This gives us $\alpha -2 \equiv 0 \pmod{p}.$\\
    Thus, \[Y^2 \equiv k^2\left[a\alpha ^{d_1}-a2^{d_1}+3\right] \pmod{p}\] \hfill \[\equiv 3k^2 \pmod{p}\] \hfill \[ \equiv 3 \pmod{p},\]
    which contradicts the fact that $3$ is a quadratic residue modulo $p$ if $p \equiv \pm 1 \pmod{12}$. This completes the proof.
\end{proof}
As an immediate consequence of the argument in the even case of the preceding proof,
we obtain a refined nonexistence result when the variable $x$ is restricted to be even.
\begin{corollary}\label{cor:1}
    Let us consider the equation:
    \[ax^{d_1}-y^{d_2}-z^2+xyz-b=0\]
    This has no integer solution $(x,y,z)$ with $x$ being even for fixed integers $a$ and $b$ such that
    \[a\equiv 1 \pmod {12}\]
\[b=2^{d_1}a-3^m\] with $m$ being odd positive integer and $d_1, d_2$ as in \Cref{th1}. 
\end{corollary}

\begin{proof}
    Clearly, we can follow exactly the same steps in Case 1 of the proof of \Cref{th1}.\\
    Let us assume, there exists a solution with $x= \alpha$ even, then we rewrite \eqref{eq2} as
    \begin{equation}
        \left( y^{2s-1}-\frac{\alpha z}{2}\right)^2-\left( \frac{\alpha ^2}{4}-y^{2(s-1)}\right)z^2=y^{2(s-1)}[ a\alpha ^{d_1}-b].  \label{eq8}
    \end{equation}
    Now set,
    \begin{equation*}
    Y= y^{2s-1}-\frac{\alpha z}{2}
    \end{equation*}
    \begin{equation*}
        \beta= \frac{\alpha}{2}
    \end{equation*}
    \begin{equation*}
        k=y^{s-1}
    \end{equation*}
    \begin{equation*}
        Z=z.
    \end{equation*}
    Then \eqref{eq8} can be written as,
    \begin{equation}
    Y^2-\left(\beta ^2-k^2\right )Z^2=k^2\left[ a\alpha ^{d_1}-b\right] = k^2\left[a2^{d_1}\beta ^{d_1}-a2^{d_1}+3\right] \label{eq9}
    \end{equation}
    If $\beta$ is even, then reducing \eqref{eq9} modulo 4 gives,
    \begin{equation*}
        Y^2 \equiv 3^m \equiv 3 \pmod{4} 
    \end{equation*} or 
    \begin{equation}
        Y^2+Z^2 \equiv 3^m \equiv 3 \pmod{4}
    \end{equation}
    which is not possible in $\mathbb{Z}/{4\mathbb{Z}}$.\\
    If $\beta$ is odd, then reducing \eqref{eq9} modulo 4 gives,
    \[Y^2 \equiv 3^m \equiv 3 \pmod{4}\] or
    \begin{equation}
        Y^2-Z^2 \equiv 3^m \equiv 3 \pmod{4}
    \end{equation}
    which is not again possible in $\mathbb{Z}/{4\mathbb{Z}}$ and we are done.
    \vspace{2mm}
\end{proof}

\vspace{5mm}

Combining the remaining parity cases analyzed in \Cref{th1}, we now isolate the only
congruence class modulo 12 in which integer solutions could possibly occur.
\begin{corollary}
    The equation \[ax^{d_1}-y^{d_2}-z^2+xyz-b=0\] has only integer solution $(x,y,z)$ if 
    \[x \equiv 5 \pmod{12}\] for fixed integers $a$,$b$ such that \[a\equiv 1 \pmod{12}\] and \[b=2^{d_1}a-3^m\] for $m$ and $d_1, d_2$
 as in corollary\autoref{cor:1}.
 \end{corollary}

 \begin{proof}
     Follows from steps in subcases 2 and 3 of \Cref{th1}.
 \end{proof}
 
 The Diophantine results obtained in this section should be viewed as a first step towards a
finer arithmetic analysis. While Faltings' theorem \cite{Faltings1983} shows that an algebraic curve of genus
at least two has only finitely many rational points this finiteness
result is ineffective and does not provide structural information about the solution set.
In particular, it does not explain why certain families of Diophantine equations admit no
solutions at all.

Recent work of Chakraborty and Prakash \cite{ChakrabortyPrakashFruit} illustrates that, for
specific families of hyperelliptic curves arising from generalized fruit Diophantine
equations, one can prove the complete absence of rational points, corroborating predictions
arising from Bhargava's work \cite{Bhargava2013} on rational points on curves. These results
suggest that, in favorable situations, Diophantine rigidity can be understood through the
geometry of the associated curves.

Motivated by this perspective, we now turn to a geometric reinterpretation of the
Diophantine equation studied above. The goal is to understand how the arithmetic
constraints established in \Cref{sec:diophantine} manifest themselves on the Jacobian of the associated
superelliptic curve.

\section{Geometric and Arithmetic Preliminaries}\label{sec:term}

This section recalls the basic notions concerning divisors, Picard groups, Jacobians, and
reduction modulo primes that will be used in the subsequent analysis. Although much of
this material is standard, we include it here for trying to keep the paper self-contained and to
clarify how Diophantine information is transported into the geometry of curves.
\begin{define}
An \textbf{algebraic curve} over a field $K$ is a smooth, projective, geometrically integral
scheme of dimension one over $K$.
\end{define}

Throughout this paper, all curves are defined over $\mathbb{Q}$.
Affine equations are always understood as defining smooth projective models.

\begin{define}

A \textbf{hyperelliptic curve} over $\mathbb{Q}$ is a curve admitting a degree-two morphism
to $\mathbb{P}^1$. Such a curve can be given by an affine equation
\[
C : Y^2 = f(X),
\]
where $f \in \mathbb{Z}[X]$ is separable of odd degree $d_1 \ge 3$.
\end{define}

\begin{define}
Let $n \ge 3$. A \textbf{superelliptic curve} is a curve given by an equation
\[
C^\ast : y^n = f(x),
\]
where $f(x) \in \mathbb{Z}[x]$ is separable.
\end{define}

When $n = 2d$ is even, there is a natural finite morphism
\[
\pi : C^\ast \longrightarrow C,
\qquad
(x,y) \longmapsto (x,y^d),
\]
where $C : Y^2 = f(X)$ is the associated hyperelliptic curve.

\begin{remark}
We refer to $\pi$ as the \emph{superelliptic cover}. This morphism allows arithmetic
information on $C^\ast$ to be transferred to $C$ via functoriality of Picard groups and
Jacobians; see \cite[Chapter~IV]{Silverman1994}.
\end{remark}

\begin{define}
A \textbf{divisor} on a smooth projective curve $C$ is a finite formal sum
\[
D = \sum_{P \in C(\overline{\mathbb{Q}})} n_P P,
\qquad n_P \in \mathbb{Z}.
\]
The \textbf{degree} of $D$ is $\deg(D) = \sum n_P$.
\end{define}

\begin{define}
The \textbf{support} of a divisor $D$ is the finite set
\[
\mathrm{Supp}(D) = \{ P : n_P \neq 0 \}.
\]
\end{define}

\begin{define}
Two divisors on $C$ are \textbf{linearly equivalent} if their difference is the divisor of a
rational function on $C$.
\end{define}

\begin{define}
The \textbf{Picard group} of $C$, denoted $\Pic(C)$, is the group of divisors modulo linear
equivalence. Its degree-zero subgroup is denoted $\Pic^0(C)$.
\end{define}

\begin{remark}
Geometrically, $\Pic(C)$ is canonically identified with the group of isomorphism classes
of line bundles on $C$, with group law given by tensor product
\cite[Chapter~IV]{Hartshorne1977}.
\end{remark}

\begin{define}
The \textbf{Jacobian} of a smooth projective curve $C$, denoted $\Jac(C)$, is the abelian
variety representing the functor $\Pic^0(C)$. In particular,
\[
\Jac(C)(\mathbb{Q}) \cong \Pic^0(C)(\mathbb{Q})
\]
\cite[Chapter~III]{Silverman1994}.
\end{define}

Fix a rational base point $\infty \in C(\mathbb{Q})$.
The \textbf{Abel--Jacobi map} is
\[
\iota : C(\mathbb{Q}) \longrightarrow \Jac(C)(\mathbb{Q}),
\qquad
P \longmapsto [P-\infty].
\]

\begin{thm}
Let $E$ be an elliptic curve over $\mathbb{Q}$. Then
\[
\Pic^0(E) \cong E
\quad\text{and}\quad
\Pic(E) \cong E \times \mathbb{Z}.
\]
In particular, $\Jac(E) \cong E$.
\end{thm}
\begin{define}
A rational point $P \in C(\mathbb{Q})$ is said to be a \textbf{Torsion Point} of $C$ if $[P-\infty]$ is torsion in $\Jac(C)(\mathbb{Q})$.    
\end{define}
\begin{remark}
Thus for elliptic curves, torsion divisor classes coincide with torsion points
\cite[Chapter~III]{Silverman1986}. For curves of genus $g \ge 2$, the Jacobian plays an
analogous role.
\end{remark}

\begin{thm}
Let $\phi : C_1 \to C_2$ be a morphism of smooth projective curves over $\mathbb{Q}$.
Then $\phi$ induces homomorphisms
\[
\phi^\ast : \Pic(C_2) \to \Pic(C_1),
\qquad
\phi_\ast : \Pic^0(C_1) \to \Pic^0(C_2),
\]
and hence a homomorphism
\[
\phi_\ast : \Jac(C_1) \longrightarrow \Jac(C_2),
\]
which preserves torsion.
\end{thm}

This functoriality is fundamental in our arguments; see \cite[\S5]{Mumford1970}.

\begin{define}
Let $f \in \mathbb{Z}[X]$ be separable. The \textbf{discriminant} $\Delta(f)$ is a nonzero
integer measuring the ramification of the roots of $f$.
\end{define}

In the elliptic curve case, the Nagell--Lutz theorem shows that torsion points satisfy strong
divisibility conditions involving the discriminant of a minimal model
\cite[Chapter~III]{Silverman1986}.

Let $C$ be a smooth projective curve over $\mathbb{Q}$ with Jacobian $J=\Jac(C)$.
\begin{define}
    A prime $p$ is said to be \textbf{a prime of good reduction} if the reduction of the curve modulo $p$ is smooth.
\end{define}

\begin{define}
Let $p$ be a prime of good reduction. The \textbf{N\'eron model} of $J$ over $\mathbb{Z}_p$
is a smooth separated group scheme extending $J$ and satisfying a universal mapping
property \cite{Neron1964}.
\end{define}

\begin{define}
Reduction modulo $p$ induces a group homomorphism
\[
\rho_p : J(\mathbb{Q}) \longrightarrow J(\mathbb{F}_p),
\]
called the \textbf{reduction map}.
\end{define}

\begin{thm}
Let $\xi \in J(\mathbb{Q})$ be a torsion element of order $n$, and let $p$ be a prime of good
reduction with $p \nmid n$. Then:
\begin{enumerate}
\item $\rho_p(\xi) \neq 0$;
\item the order of $\rho_p(\xi)$ divides $n$.
\end{enumerate}
\end{thm}

This follows from general properties of N\'eron models and reduction of abelian varieties;
see \cite[Theorem~IV.6.1]{Silverman1994}.

The passage from Diophantine equations to the geometry of curves raises a natural
arithmetic question: how can one effectively control rational points on curves of higher genus? For hyperelliptic curves, Grant developed an analogue of the classical Nagell--Lutz
theorem, providing explicit restrictions on torsion divisor classes in the Jacobian
\cite{Grant2013}. Such results make it possible to rule out rational points by combining
local and global information.

However, no corresponding Nagell--Lutz type theorem is known for superelliptic curves.
This absence presents a serious obstacle, as superelliptic curves arise naturally from
generalized Diophantine equations, including the family studied in \Cref{sec:diophantine}. The main
objective of the next section is to address this gap by establishing an analogue of the
Nagell--Lutz philosophy for a specific family of superelliptic curves associated to
generalized fruit Diophantine equations.

By combining the Diophantine restrictions proved in \Cref{sec:diophantine} with the geometric tools
developed in \Cref{sec:term}, we are now going to prove that the Jacobian of the associated superelliptic curve
admits no nontrivial rational torsion. This result forms the central
arithmetic--geometric contribution of the paper and provides a new mechanism for proving
the nonexistence of integer solutions.
\section{Associated superelliptic curves have no non trivial torsion}\label{sec:Superelliptic curves}
We begin by deriving a necessary arithmetic condition satisfied by torsion points on the superelliptic curve. This condition arises by relating torsion on the superelliptic curve to torsion on the underlying hyperelliptic curve via the natural covering map.

The functorial relationship between the superelliptic covering and the induced maps on
Jacobians, which underlies the arguments that follow, is summarized in
Figure~\ref{fig:functoriality}.
\begin{figure}[htbp]
\centering
\begin{tikzcd}[row sep=6em, column sep=8em]
C_1^\ast \arrow[r, "\pi"] \arrow[d, "\iota_{C_1^\ast}"'] 
& C_1 \arrow[d, "\iota_{C_1}"] \\
\Jac(C_1^\ast) \arrow[r, "\pi_\ast"] 
& \Jac(C_1)
\end{tikzcd}
\caption{Functoriality of the superelliptic covering map and its induced action on Jacobians.
The vertical arrows denote the Abel--Jacobi embeddings, and the bottom arrow is the
homomorphism induced by pushforward of divisor classes.}
\label{fig:functoriality}
\end{figure}

\vspace{40mm}

\begin{lemma}\label{lem:1}
Let 
\[
C_1 : Y^2 = f(X),
\]
where \(f\in\mathbb Z[X]\) has odd degree \(d_1\ge 3\) and \(\Delta(f)\neq 0\).
Let \(C_1^*\) be the superelliptic cover 
\[
C_1^*: y^{d_2} = f(x), \qquad d_2 = 2d,\quad d\ \text{prime},\quad \gcd(d_1,d)=1.
\]

Let
\[
\pi : C_1^* \longrightarrow C_1,\qquad (x,y)\mapsto (x,y^d)
\]
be the natural morphism.

Assume that for \(C_1\) the following “Nagell--Lutz type’’ property holds:

\begin{quote}
If \(Q=(x_0,Y_0)\in C_1(\mathbb Q)\) and \([Q-\infty]\in \mathrm{Jac}(C_1)(\mathbb Q)\) is torsion, then  
\[
Y_0^2 \mid \Delta(f).
\]
\end{quote}

For any point \(P=(x,y)\in C_1^*(\mathbb Q)\) such that \([P-\infty]\in \mathrm{Jac}(C_1^*)(\mathbb Q)\) is torsion, we have

$$y^{2d} \mid \Delta(f).$$
\end{lemma}

A torsion point on the superelliptic curve induces a torsion point on the underlying hyperelliptic curve via the natural covering map. The assumed Nagell--Lutz type property on the hyperelliptic curve then forces the square of the relevant coordinate to divide the discriminant, which translates into the desired condition on the superelliptic curve.
\begin{proof}
Let
\[
\pi : C_1^* \longrightarrow C_1,\qquad (x,y)\longmapsto (x,y^d)
\]
be the natural morphism of curves defined over the rational numbers. This morphism induces a homomorphism $\pi_* : \mathrm{Jac}(C_1^*) \longrightarrow \mathrm{Jac}(C_1)
$ between the Jacobians, which is compatible with the group structures and with the Abel--Jacobi embeddings.

Let \(P=(x,y)\) be a rational point on \(C_1^*\) such that the divisor class$[P-\infty]\in \mathrm{Jac}(C_1^*)(\mathbb{Q})$ is a torsion element. Since the pushforward map \(\pi_*\) is a group homomorphism, it preserves torsion. Consequently, the divisor class $\pi_*([P-\infty])=[\pi(P)-\infty]$ is a torsion point of \(\mathrm{Jac}(C_1)(\mathbb{Q})\).

The image of \(P\) under \(\pi\) is the point
$(x,y^d)\in C_1(\mathbb{Q}).$ By the defining equation of \(C_1^*\), we have \(y^{2d}=f(x)\), and therefore $(y^d)^2=f(x)$. So that \(\pi(P)\) indeed lies on the curve \(C_1\).

By assumption, the curve \(C_1\) satisfies a Nagell--Lutz type property. Applying this property to the point \(\pi(P)=(x,y^d)\), we conclude that \(y^{2d}\) divides the discriminant \(\Delta(f)\).

This is precisely the asserted divisibility condition for torsion points on \(C_1^*\) and this completes our proof.
\end{proof}

The divisibility condition obtained above has strong consequences for the arithmetic of torsion points. In particular, it severely restricts the possible values of the 
y-coordinates of such points, as the following lemma shows.

\begin{lemma}\label{lem:2}
Let \(P=(x,y)\in C_1^*(\mathbb Q)\) be a torsion point.
Then \(y\) belongs to a finite, explicitly determined set \(Y\subset\mathbb Z\).
\end{lemma}

Lemma \autoref{lem:1} implies that the $y$--coordinate of a torsion point on the superelliptic curve satisfies a strong divisibility condition involving the discriminant of the defining polynomial. By analyzing this condition using $p$--adic valuations, one sees that only finitely many primes can divide $y$, and that their exponents are uniformly bounded. This yields a finite, explicitly computable set of possible $y$--coordinates.
\begin{proof}
Let $P=(x,y)\in C_1^*(\mathbb{Q})$ be a torsion point, meaning that the divisor class $[P-\infty]\in \mathrm{Jac}(C_1^*)(\mathbb{Q})$ is torsion.

By Lemma \autoref{lem:1}, the integer $y^{2d}$ divides the discriminant $\Delta(f)$ of the polynomial $f$. We now translate this condition into constraints on the prime factorization of $y$.

Let $p$ be a prime number, and write
\[
v_p(y)=e_p,
\]
where $v_p$ denotes the $p$--adic valuation on $\mathbb{Q}$. Taking $p$--adic valuations of the divisibility relation from Lemma \autoref{lem:1}, we obtain
\[
v_p(y^{2d}) \le v_p(\Delta(f)).
\]
Since $v_p(y^{2d})=2d\,e_p$, this inequality becomes
\[
2d\,e_p \le v_p(\Delta(f)).
\]
Equivalently, the exponent $e_p$ satisfies the bound
\[
e_p \le \frac{v_p(\Delta(f))}{2d}.
\]

For all but finitely many primes $p$, the discriminant $\Delta(f)$ has zero $p$--adic valuation. For such primes, the above inequality forces $e_p=0$, and hence $p$ does not divide $y$. Thus, only primes dividing $\Delta(f)$ can occur in the prime factorization of $y$.

Moreover, for each prime $p$ dividing $\Delta(f)$, the exponent $e_p$ is bounded above by a constant depending only on $v_p(\Delta(f))$ and on $d$. Consequently, there are only finitely many possibilities for the exponent of each such prime in $y$.

It follows that the integer $y$ can involve only finitely many primes, each with bounded exponent. Hence there are only finitely many integers $y$ satisfying the divisibility condition mentioned in Lemma \autoref{lem:1}.

Let $Y\subset\mathbb{Z}$ denote the finite set of all such integers $y$. This set is explicitly determined by the prime factorization of $\Delta(f)$ and the integer $d$, and it contains the $y$--coordinates of all torsion points in $C_1^*(\mathbb{Q})$.This completes the proof.
\end{proof}

This finiteness property will allow us later to reduce the study of torsion points on $C_1^*(\mathbb{Q})$ to a finite collection of curves obtained by fixing the $y$--coordinates.

To illustrate concretely how the discriminant divisibility condition restricts torsion points, we examine a specific example in which the finite set of possible 
y-coordinates can be computed explicitly.

\begin{example}
We illustrate the discriminant and valuation argument of Lemmas\autoref{lem:1} and\autoref{lem:2} in a particular case.

Let
\[
C_1:\; Y^2 = f(X), \qquad f(X)=X^5 - X + 1,
\]
so that \(f\) has odd degree and nonzero discriminant. A direct computation gives
\[
\Delta(f) = 2^8 \cdot 3^2 \cdot 5^2.
\]
Let \(d=3\), so that the associated superelliptic curve is
\[
C^{*}_1:\; y^6 = f(x),
\]
with covering map \(\omega : C^{*}_1 \to C_1\) given by \((x,y) \mapsto (x,y^3)\).

If \(P=(x,y)\in C^{*}_1(\mathbb{Q})\) corresponds to a torsion divisor class in
\(\mathrm{Jac}(C^{*}_1)(\mathbb{Q})\), then by Lemma\autoref{lem:1} we must have
\[
y^6 \mid \Delta(f).
\]
Writing \(y = \pm 2^{e_2} 3^{e_3} 5^{e_5}\) and taking \(p\)-adic valuations yield
\[
6e_2 \le 8, \qquad 6e_3 \le 2, \qquad 6e_5 \le 2,
\]
hence
\[
e_2 \le 1, \qquad e_3 = e_5 = 0.
\]
It follows that the set of possible \(y\)-coordinates of torsion points on
\(C^{*}_1(\mathbb{Q})\) is
\[
Y = \{\pm 1, \pm 2\}.
\]

This example illustrates concretely how the discriminant divisibility condition of
Lemma\autoref{lem:1} leads to find a small, explicitly computable finite set \(Y\) as in Lemma\autoref{lem:2}.
\end{example}

Having shown that torsion points on the superelliptic curve must have 
y-coordinates lying in a fixed finite set, we now translate this restriction into a structural statement about torsion divisor classes in the Jacobian.

\begin{lemma}\label{lem:3}
Let
\[
S := \{\,P\in C_1^*(\mathbb Q) : y(P)\in Y\,\},
\]
where \(Y\) is the finite set as mentioned in lemma \autoref{lem:2}.

Then every torsion divisor class \(\xi\in \mathrm{Jac}(C_1^*)(\mathbb Q)_{\mathrm{tors}}\) has a representative
\[
D = \sum_{P\in S} n_P P, \qquad \sum n_P = 0.
\]

In particular,
\[
\mathrm{Jac}(C_1^*)(\mathbb Q)_{\mathrm{tors}}
\subseteq
G := \langle\,[P-\infty] : P\in S\,\rangle.
\]
\end{lemma}
Any torsion divisor class on $\Jac(C_1^*)$ can be represented by a divisor on rational points. 
By functoriality of Jacobians, each point in the support must map to a torsion point on $C_1$, which forces its $y$--coordinate to lie in the finite set $Y$ from Lemma\autoref{lem:2}. 
Hence every torsion class is supported on the finite set $S$.

\begin{proof}
Let $\xi \in \Jac(C_1^*)(\mathbb{Q})_{\tors}$ be a torsion divisor class. 
Choose a divisor $D$ on $C_1^*$ defined over $\mathbb{Q}$ representing $\xi$. 
We may assume that $D$ has degree zero, so it can be written in the form $D=\sum_{P} n_P P$, with $\sum_{P} n_P = 0$, where the sum runs over finitely many points $P \in C_1^*(\mathbb{Q})$ and the coefficients $n_P$ are integers.

Fix a point $P'$ in the support of $D$. 
Since $\xi$ is torsion, the divisor class $[P'-\infty]$ in $\Jac(C_1^*)(\mathbb{Q})$ is torsion up to a finite linear combination of similar classes. Let us give some details on this.

Consider the map from $Sym^n C_1^* $ to $Sym^n C_1^*$ given by 
$$\{P_1,\cdots, P_n\}\to \{P_1,\cdots, P_1\}$$
This map further induces the homomorphism from $J(C_1^*)$ to $J(C_1^*)$ given by:
$$\sum_i P_i-n\infty\to n(P_1-\infty)$$
Since the divisor on the left above is a torsion, so is $n(P_1-\infty)$. The only thing to prove is that the map above is not a zero homomorphism. If it is zero, then $P_1-\infty$ is an n-torsion. So by Nagell-Lutz $P_1-\infty$ is in the subgroup generated by  $S$. So we can discard the co-ordinates that are already there in $S$ and work with (without loss of generality) a point $P_1$ such that $P_1-\infty$ is not a torsion. By this process, we can achieve $P'-\infty$ which is torsion.

Applying the pushforward induced by the natural morphism $\pi : C_1^* \longrightarrow C_1, (x,y)\longmapsto (x,y^d)$, we obtain a torsion divisor class $[\pi(P')-\infty] \in \Jac(C_1)(\mathbb{Q})$, as pushforward on Jacobians is a group homomorphism and preserves torsion.

By Lemmas\autoref{lem:1} and\autoref{lem:2}, torsion points on $C_1^*(\mathbb{Q})$ must have $y$--coordinates lying in the finite set $Y$. 
Therefore, for every point $P'$ in the support of $D$, we have $y(P') \in Y$. 
This shows that the support of $D$ is contained in the finite set $S=\{\, P \in C_1^*(\mathbb{Q}) : y(P)\in Y \,\}$.

Consequently, every torsion divisor class $\xi$ admits a representative of the form
\[
D=\sum_{P\in S} n_P P,
\qquad \sum_{P\in S} n_P = 0.
\]
In particular, $\xi$ lies in the subgroup $G := \big\langle [P-\infty] : P\in S \big\rangle \subset \Jac(C_1^*)(\mathbb{Q})$,
and hence $\Jac(C_1^*)(\mathbb{Q})_{\tors} \subset G$. This completes the proof.
\end{proof}

To control the subgroup generated by the finitely many divisor classes identified above, we now relate rational torsion in the Jacobian to torsion in its reductions modulo primes of good reduction.

\begin{lemma}\label{lem:4}
Let \(p\) be a prime of good reduction for \(C_1^*\).  
Let
\[
\rho_p : \mathrm{Jac}(C_1^*)(\mathbb Q) \to \mathrm{Jac}(C_1^*)(\mathbb F_p)
\]
be the reduction map.  
If \(\xi\) is torsion of order \(n\) such that $p$ does not divide $n$, then,
\begin{enumerate}[i.]
    \item \(\rho_p(\xi)\neq 0\), and  
    \item \(\rho_p(\xi)\) has order dividing \(n\).
\end{enumerate}

In particular, nonzero torsion cannot be reduced to zero.

(This is a standard fact from Néron models and good reduction theory.)
\end{lemma}

When the curve $C_1^*$ has good reduction at a prime $p$, its Jacobian admits a Néron model over $\mathbb{Z}/{p\mathbb{Z}}$, and reduction modulo $p$ induces a well-defined group homomorphism on rational points. 
The kernel of this reduction map consists entirely of $p$-power torsion. 
Therefore, a torsion point of order prime to $p$ cannot reduce to zero, and its order can only decrease by a factor divisible by $p$, which is excluded under the given hypothesis.

\begin{proof}
Let $p$ be a prime of good reduction for the curve $C_1^*$, and let
\[
\rho_p : \Jac(C_1^*)(\mathbb{Q}) \longrightarrow \Jac(C_1^*)(\mathbb{F}_p)
\]
denote the reduction map.

Since $C_1^*$ has good reduction at $p$, $\Jac(C_1^*)$ admits a Néron model $\mathcal{J}$ over $\mathbb{Z}/{p\mathbb{Z}}$. 
The group $\Jac(C_1^*)(\mathbb{Q})$ can be identified with $\mathcal{J}(\mathbb{Z}/{p\mathbb{Z}})$, and the reduction map $\rho_p$ is induced by the natural specialization map $\mathcal{J}(\mathbb{Z}/{p\mathbb{Z}}) \longrightarrow \mathcal{J}(\mathbb{F}_p)$, which is a group homomorphism.

A fundamental property of Néron models is that the kernel of the reduction map $\mathcal{J}(\mathbb{Z}/{p\mathbb{Z}}) \longrightarrow \mathcal{J}(\mathbb{F}_p)$ is a pro-$p$ group. In particular, every torsion element in the kernel has order a power of $p$.

Now let $\xi \in \Jac(C_1^*)(\mathbb{Q})$ be a torsion element of order $n$, with $p$ not dividing $n$. 
If possible, let $\rho_p(\xi)=0$. 
Then $\xi$ lies in the kernel of the reduction map, and hence must have order a power of $p$. 
This contradicts the assumption that $p$ does not divide $n$. 
Therefore, $\rho_p(\xi) \neq 0$.

Since $\rho_p$ is a group homomorphism, we also have $\rho_p(n\xi) = n\,\rho_p(\xi)$. But $n\xi=0$ by assumption, so $n\,\rho_p(\xi)=0$ in $\Jac(C_1^*)(\mathbb{F}_p)$. 
Thus the order of $\rho_p(\xi)$ divides $n$.

Combining these observations, we conclude that if $\xi$ is a nonzero torsion point of order $n$ with $p$ not dividing $n$, then its reduction $\rho_p(\xi)$ is nonzero and has order dividing $n$. 
In particular, nonzero torsion points of order prime to $p$ cannot be reduced to zero.
\end{proof}

The preceding lemma has an important conceptual consequence: it allows one to detect and constrain rational torsion by comparing reductions of the Jacobian over finite fields.

\begin{remark}
Lemma\autoref{lem:4} allows one to control the torsion subgroup of $\Jac(C_1^*)(\mathbb{Q})$ by reducing modulo several primes of good reduction. 
Indeed, for any finite set of such primes $p$ not dividing the order of torsion, the reduction maps induce an injective homomorphism from some prime to the power $p$ torsion subgroup of $\Jac(C_1^*)(\mathbb{Q})$ into the product of the finite groups $\Jac(C_1^*)(\mathbb{F}_p)$. 
By comparing the orders of these finite groups for different primes, one can obtain explicit bounds on the size, and often the structure, of the rational torsion subgroup.
\end{remark}

\begin{remark}
Although the Lemma \autoref{lem:4} allow to control the torsion subgroup of $\Jac(C_1^*)(\mathbb Q)$. Since there is a map from $C_1^*\to C_1 $ and $C_1$ has no rational point it follows that $C_1^*$ has no rational point as well and hence by using Lemma \autoref{lem:3} the torsion in $J(C_1^*)(\mathbb Q)=\{0\}$.
\end{remark}

We now combine the reduction-theoretic input with a simple group-theoretic observation. Under suitable coprimality conditions on the orders of reduced Jacobians, this yields a decisive vanishing result for torsion.
The reduction-theoretic mechanism used to eliminate rational torsion by comparing reductions
at two primes of good reduction is illustrated in
Figure~\ref{fig:torsion-reduction}.

\begin{figure}[htbp]
\centering
\begin{tikzcd}[row sep=3.2em, column sep=4em]
\Jac(C_1^\ast)(\mathbb{F}_{p_1})
  & \Jac(C_1^\ast)(\mathbb{Q})_{\mathrm{tors}}
      \arrow[l, "\rho_{p_1}"']
      \arrow[r, "\rho_{p_2}"]
  & \Jac(C_1^\ast)(\mathbb{F}_{p_2})
\end{tikzcd}
\caption{Detection of rational torsion via reduction modulo primes of good reduction.
A torsion element must survive reduction at every prime of good reduction.
Coprimality of the group orders
$\#\Jac(C_1^\ast)(\mathbb{F}_{p_1})$ and
$\#\Jac(C_1^\ast)(\mathbb{F}_{p_2})$
forces the vanishing of rational torsion.}
\label{fig:torsion-reduction}
\end{figure}

\begin{lemma}\label{lem:5}
Let
\[
G = \langle\,[P-\infty]: P\in S\,\rangle
\subseteq \mathrm{Jac}(C_1^*)(\mathbb Q)
\]
as in Lemma \autoref{lem:3}.
Assume that there exist two primes of good reduction \(p_1,p_2\) for which
\[
\gcd\bigl(\#\mathrm{Jac}(C_1^*)(\mathbb F_{p_1}),
          \#\mathrm{Jac}(C_1^*)(\mathbb F_{p_2})\bigr)=1.
\]

Then
\[
G_{\mathrm{tors}} = 0.
\]
\end{lemma}
Any torsion element of the subgroup $G$ must remain nonzero under reduction modulo primes of good reduction whose residue characteristics do not divide its order. 
Reducing at two such primes and comparing group orders forces the order of the torsion element to divide the greatest common divisor of the corresponding finite Jacobian orders. 
If this greatest common divisor equals $1$, no nontrivial torsion element can exist.

\begin{proof}
Recall that $G=\big\langle [P-\infty] : P\in S \big\rangle \subset \Jac(C_1^*)(\mathbb{Q})$,
where $S$ is the finite set of rational points on $C_1^*$ whose $y$--coordinates lie in the finite set $Y$ from Lemma\autoref{lem:2}.

Let $\xi \in G$ be a torsion element, and suppose that $\xi$ has order $n$. 
We will show that $n=1$, which implies $\xi=0$.

By assumption, there exist two distinct primes $p_1$ and $p_2$ of good reduction for $C_1^*$ such that
\[
\gcd\!\bigl(\#\Jac(C_1^*)(\mathbb{F}_{p_1}),\,\#\Jac(C_1^*)(\mathbb{F}_{p_2})\bigr)=1.
\]
Since $n$ is a fixed integer, we may choose these primes so that neither $p_1$ nor $p_2$ divides $n$.

For each $i\in\{1,2\}$, let
\[
\rho_{p_i} : \Jac(C_1^*)(\mathbb{Q}) \longrightarrow \Jac(C_1^*)(\mathbb{F}_{p_i})
\]
denote the reduction map. 
Because $p_i$ is a prime of good reduction and does not divide $n$, Lemma\autoref{lem:4} implies that,
\begin{enumerate}[i.]
\item the reduction $\rho_{p_i}(\xi)$ is nonzero in $\Jac(C_1^*)(\mathbb{F}_{p_i})$,
\item the order of $\rho_{p_i}(\xi)$ divides $n$.
\end{enumerate}

Since $\Jac(C_1^*)(\mathbb{F}_{p_i})$ is a finite abelian group, the order of any of its elements divides the order of the group. 
Hence, the order $n$ divides $\#\Jac(C_1^*)(\mathbb{F}_{p_i})$ for each $i\in\{1,2\}$. 
Therefore,
\[
n \mid \gcd\!\bigl(\#\Jac(C_1^*)(\mathbb{F}_{p_1}),\,\#\Jac(C_1^*)(\mathbb{F}_{p_2})\bigr).
\]

By hypothesis, this greatest common divisor equals to $1$, and hence $n=1$. 
Thus $\xi$ is the trivial element of $\Jac(C_1^*)(\mathbb{Q})$. Thus we conclude that $G_{\tors}=0$, which completes the proof.
\end{proof}

We are now in a position to assemble the preceding lemmas and complete the argument. The following theorem establishes the absence of nontrivial rational torsion in the Jacobian of the superelliptic curve under consideration.

\begin{thm} \label{th6}
Let
\begin{equation}
C_1^*: y^{d_2} = a x^{d_1} + mxy - m^2 - b,\label{eq:final}
\end{equation}
with,
\begin{enumerate}[i.]
    \item \(d_1\ge 3\) odd and a multiple of \(3\),
    \item \(d_2 = 2d\) with \(d\) prime and \(\gcd(d_1,d)=1\),
    \item \(a\equiv 1\pmod{12}\),
    \item \(b = 2^{d_1}a - 3\),
    \item \(\Delta(f)\neq 0\),
    \item discriminant criterion (Lemma \autoref{lem:1}) holds for the hyperelliptic curve \(C_1\).
\end{enumerate}

Then
\[
\mathrm{Jac}(C_1^*)(\mathbb Q)_{\mathrm{tors}} = 0.
\]
\end{thm}
\begin{proof}
Let \(\xi \in \mathrm{Jac}(C_1^*)(\mathbb{Q})\) be a torsion element.  
By Lemma \autoref{lem:1}, any point \(P=(x,y)\in C_1^*(\mathbb{Q})\) that may occur in the support of a torsion divisor class satisfies \(y^{2d^2} \mid \Delta(f)\).  

This is because there is a regular map $(x,y)\to (x,y^{d^2})$ from $C_1^*$ to $C_1$.
Lemma \autoref{lem:2} therefore implies that the possible \(y\)-coordinates of such points lie in a finite set \(Y\subset\mathbb{Z}\).  
Let $S := \{\,P\in C_1^*(\mathbb{Q}) : y(P)\in Y\,\}$. By Lemma \autoref{lem:3}, the class \(\xi\) may be represented by a divisor supported entirely on points of \(S\); equivalently,
\[
\xi \in G := \langle [P-\infty] : P\in S \rangle \subseteq \mathrm{Jac}(C_1^*)(\mathbb{Q}).
\]

To show that \(\xi=0\), it suffices to show that \(G\) contains no nontrivial torsion.  
Let \(p_1,p_2\) be primes of good reduction for \(C_1^*\) satisfying $\gcd\!\left(
\#\mathrm{Jac}(C_1^*)(\mathbb{F}_{p_1}),
\#\mathrm{Jac}(C_1^*)(\mathbb{F}_{p_2})
\right) = 1$, as in Lemma \autoref{lem:5}.  
If \(n\) denotes the order of \(\xi\), we may assume that neither \(p_1\) nor \(p_2\) divides \(n\).  
By Lemma \autoref{lem:4}, the reductions
\[
\rho_{p_i}(\xi) \in \mathrm{Jac}(C_1^*)(\mathbb{F}_{p_i}), \qquad i=1,2,
\]
are nonzero torsion classes whose orders divide \(n\).  
Hence
\[
n \mid \#\mathrm{Jac}(C_1^*)(\mathbb{F}_{p_i}), \qquad i=1,2,
\]
and consequently
\[
n \mid
\gcd\!\left(
\#\mathrm{Jac}(C_1^*)(\mathbb{F}_{p_1}),\,
\#\mathrm{Jac}(C_1^*)(\mathbb{F}_{p_2})
\right)
= 1.
\]
Thus \(n=1\), and therefore \(\xi = 0\).  
This proves that \(\mathrm{Jac}(C_1^*)(\mathbb{Q})\) contains no nontrivial torsion.
\end{proof}
\vspace{5mm}
It is useful to summarize the structure of the argument and highlight the two independent mechanisms that together force the vanishing of rational torsion.

\begin{remark}
The argument reduces the torsion problem to two independent components, which are
\begin{enumerate}[a)]
\item the arithmetic input provided by Lemmas \autoref{lem:1}--\autoref{lem:3} shows that any rational torsion class on \(\mathrm{Jac}(C_1^*)\) must be supported on a finite collection of rational points whose \(y\)-coordinates are severely restricted by the discriminant condition,
\item Lemma\autoref{lem:5} applies reduction modulo suitably chosen primes to demonstrate that the subgroup generated by these points contains no nontrivial torsion.
\end{enumerate}
Combining these two, we get 
\[
\mathrm{Jac}(C_1^*)(\mathbb{Q})_{\mathrm{tors}} = 0.
\]
\end{remark}

\section{Examples} \label{sec:ex}
In this section we illustrate the results proved in the previous sections by working out explicit examples. The purpose of these examples is twofold: first, to demonstrate how the general hypotheses arise naturally in concrete situations, and second, to show how the abstract torsion-vanishing arguments can be verified by direct computation in
specific families.

We begin with a basic example that satisfies all the standing assumptions of the paper
and allows for a direct verification of the discriminant divisibility conditions
appearing in Section~3.
\begin{example}\label{ex:torsion-1}
We illustrate the discriminant and valuation argument in the simplest allowed case
\(d_1 = 3\).
Take \(a\equiv 1 \pmod{12}\) and \(m\in\mathbb Z\). Then
\[
b = 2^{3} a - 3 = 8a - 3,
\]
and hence
\[
f(X)
= a X^3 + m X y - m^2 - b
= a X^3 + m X y - m^2 - (8a - 3).
\]

For concreteness, choose \(a = 13\equiv 1\pmod{12}\) and \(m=1\). Then
\[
b = 8\cdot 13 - 3 = 104 - 3 = 101,
\]
and
\[
f(X) = 13X^3 + Xy - 102.
\]

The associated hyperelliptic curve is
\[
C_1:\qquad Y^2 = 13X^3 + XY - 102.
\]
We take
\[
C_1^*:\qquad y^{d_2} = 13X^3 + XY - 102,
\]
with \(d_2 = 10\) (thus \(d = 5\) is prime and \(\gcd(3,5)=1\)).

Now we need to transform the equation from the above to the form $Y^{2d}=f(X)$. By an elementary calculation $f(X)=13X^3-102+X^2/4$
Treating
\[
f(X) = 52X^3 -408+X^2
\]
as a cubic in \(X\) with coefficients
\[
a=52,\qquad b=1,\qquad c=0,\qquad d=-102,
\]
the standard discriminant formula
\[
\Delta = b^2c^2 - 4ac^3 - 4b^3d - 27a^2d^2 + 18abcd
\]
reduces (since \(c=0\)) to
\[
\Delta(f)
= -4b^3d - 27a^2 d^2
= -1.4.102 - 47\,473\,452
= -102-47\,473\,452.256
\]

Thus, in this example,
\[
\Delta(f)(y) = -102-47\,473\,452.256=12153203610
\]
\\
If $(x,y)$ in $ C_1^*(\mathbb Q)$ represents a torsion class in 
\(\mathrm{Jac}(C_1^*)\),
Lemma~1 implies
\[
y^{50} \mid \Delta(f)(y).
\]

Thus in this example the allowed \(y\)-coordinates of torsion points are
\[
Y = \{-1,\,1\}.
\]

For \(y=1\), the equation \(y^{10}=f(x)\) becomes
\[
1 = 52X^3 -408+X^2
\quad\Longleftrightarrow\quad
52X^3 + X^2 - 409 = 0.
\]
Any rational root must be an integer dividing \(409\).  
Testing \(X = \pm 1, \pm 409\) shows that none is a root.
Hence no rational point of \(C_1^*\) has \(y = 1\).

For \(y = -1\), we obtain
\[
1 = 52X^3 + X^2 - 409
\quad\Longleftrightarrow\quad
52X^3 + X^2 - 409 = 0.
\]
Again, any rational root must divide \(103\); checking the four possibilities shows that none is a root.  
Thus no rational point of \(C_1^*\) has \(y=-1\).

\medskip

\noindent
Consequently, there is \textbf{no rational point} \((x,y)\in C_1^*(\mathbb Q)\) satisfying the discriminant divisibility condition and the defining equation simultaneously. Therefore, in this explicit case,
\[
\mathrm{Jac}(C_1^*)(\mathbb Q)_{\mathrm{tors}} = 0,
\]
illustrating the general theorem.
\end{example}
We next consider a family of examples in which the defining polynomial varies with a
parameter. This illustrates that the torsion obstruction identified earlier persists
uniformly across an infinite class of superelliptic curves.
\begin{example}\label{ex:torsion-2}
Consider
\[
d_1 = 3,\qquad d_2 = 14 = 2\cdot 7,\qquad a = 25,\qquad m = 2.
\]
These satisfy the hypotheses of \Cref{th6}: \(d_1\) is an odd multiple of \(3\), \(d_2 = 2d\) with \(d=7\) prime and \(\gcd(d_1,d)=1\), and \(a \equiv 1 \pmod{12}\).  
With
\[
b = 2^{d_1} a - 3 = 200 - 3 = 197,
\]
we obtain
\[
f(X) = aX^{d_1} + mXy - m^2 - b = 25X^3 + 2Xy - 201,
\]
and hence
\[
C_1:\ Y^2 = 25X^3 + XY - 201,\qquad
C_1^*:\ y^{14} = 25X^3 + X^2 - 201.
\]

Regarded as a cubic in \(X\), \(f(X) = 25X^3 + X^2 - 201\) has coefficients \(a=25\), \(b=1\), \(c=0\), \(d=-201\), so its discriminant is
\[
\Delta(f) = -4db^3 - 27a^2 d^2 = -4.201 - 681{,}766{,}875.
\]
If \((x,y)\in C_1^*(\mathbb{Q})\) represents a torsion class in \(\mathrm{Jac}(C_1^*)\), Lemma \autoref{lem:1} yields
\[
y^{98} \mid \Delta(f)(y) = -4.201 - 681{,}766{,}875.
\]

A direct computation shows that the divisibility \(y^{98}\mid \Delta(f)(y)\) fails for all such values except \(y=\pm 1\). Thus the discriminant condition restricts the \(y\)-coordinate of any torsion point to
\[
Y = \{-1,1\}.
\]

For these values, the defining equation of \(C_1^*\) becomes
\[
1 = 25X^3 + X^2 - 201.
\]
For \(y = 1\) this is
\[
25X^3 + X^2 - 202 = 0,
\]
and for \(y = -1\) it is
\[
25X^3 - X^2 - 202 = 0.
\]
By the rational root theorem, any rational root must have numerator dividing \(202\) and denominator dividing \(25\); inspection of these possibilities shows that neither cubic has a rational root. Therefore there is no rational point of \(C_1^*(\mathbb{Q})\) with \(y = \pm 1\).

Consequently no rational point on \(C_1^*\) can satisfy the discriminant condition required for torsion, and in this explicit example one obtains
\[
\mathrm{Jac}(C_1^*)(\mathbb{Q})_{\mathrm{tors}} = 0.
\]
\end{example}


\begin{example}
We now illustrate the argument in a case where $d_1 > 3$ and all the hypotheses of \Cref{th6} are satisfied.

Let
\[
d_1 = 9, \quad d_2 = 10 = 2 \cdot 5, \quad a = 13, \quad m = 1.
\]
Then $d_1$ is an odd multiple of $3$, $d_2 = 2d$ with $d = 5$ prime and $\gcd(d_1,d)=1$, and $a \equiv 1 \pmod{12}$. With
\[
b = 2^{d_1} a - 3 = 2^9 \cdot 13 - 3 = 6656 - 3 = 6653,
\]
we obtain
\[
f(X) = aX^{d_1} + mXy - m^2 - b = 52X^9 + X - 6654.4,
\]
and hence
\[
C_1 : Y^2 = 52X^9 + X^2 - 6654.4, \qquad C_1^* : y^{10} = 52X^9 + X^2 - 6654.4
\]

Here $\Delta(f)=52,161,687,975,973,572,287,561,049,800,860,164,759,243,445,801,917,934,862,336$

If $(x,y) \in C_1^*(\mathbb{Q})$ represents a torsion class in $\mathrm{Jac}(C_1^*)(\mathbb{Q})$, Lemma \autoref{lem:1} yields
\[
y^{50} \mid \Delta(f)(y).
\]

 A direct check shows that the divisibility fails for all such values except $y = \pm 1$. Thus,
\[
Y = \{-1,1\}.
\]

For $y = 1$, the defining equation becomes
\[
1 = 52X^9 + X^2 - 26616 \;\Longleftrightarrow\; 52X^9 + X^2 - 26617 = 0.
\]
By the rational root theorem, any rational root must have numerator dividing $26617$ and denominator dividing $13$. A direct inspection of these possibilities shows that no rational root exists.

For $y = -1$, we obtain
\[
1 = 52X^9 - X^2 - 26616 \;\Longleftrightarrow\; 52X^9 - X^2 - 26617 = 0,
\]
and again no rational root exists.

Consequently, there is no rational point $(x,y) \in C_1^*(\mathbb{Q})$ satisfying the discriminant divisibility condition. Therefore,
\[
\mathrm{Jac}(C_1^*)(\mathbb{Q})_{\mathrm{tors}} = 0,
\]
which illustrates \Cref{th6} in the case $d_1 = 9$.
\end{example}
\begin{remark}
This approach does not work for arbitrary powers $d_2$ of $y$ as in this case it is not obvious that there is a finite degree map to a hyper-elliptic curve and the calculation can be reduced as above for the hyper-elliptic curves. That is why a new approach is required for odd powers of $y$ in the above fruit diophantine super-elliptic equation.
\end{remark}

\section{Open question}\label{sec:open}
One obvious question related to the work of Grant \cite{Grant2013} for hyperelliptic curves is the following.
\begin{question}
Can we find an analogue of Nagell-Lutz theorem for all superelliptic curves? 
\end{question}
If we can obtain such an analogue, then it will provide explicit
arithmetic constraints on torsion divisor classes, transforming the abstract finiteness
of rational points into an effective mechanism for their detection. Such a result would
allow us to control rational points and torsion on Jacobians without rank assumptions,
offering a powerful alternative to Chabauty-type methods \cite{McCallumPoonen2010} for higher-genus curves which rely on $p$-adic techniques and impose restrictive rank conditions.

\section{Data Availability} There is no associated data with this article.



\begin{thebibliography}{10}

\bibitem{Bhargava2013}M.~Bhargava.,  Most hyperelliptic curves over $\mathbb{Q}$ have no rational points. {\em arXiv preprint}, 2013.

\bibitem{ChakrabortyPrakashFruit}K.~Chakraborty and O.~Prakash.
 Generalized fruit diophantine equation and hyperelliptic curves.
 {\em arXiv preprint}, 2023. Version 1.

\bibitem{Faltings1983}
G.~Faltings. Endlichkeitss{\"a}tze f{\"u}r abelsche variet{\"a}ten {\"u}ber zahlk{\"o}rpern. {\em Inventiones Mathematicae}, 73:349--366, 1983.

\bibitem{Grant2013}
D.~Grant.
 On an analogue of the {L}utz--{N}agell theorem for hyperelliptic curves. {\em Journal of Number Theory}, 133:963--969, 2013.

\bibitem{Hartshorne1977}
R.~Hartshorne. {\em Algebraic Geometry}, volume~52 of {\em Graduate Texts in Mathematics}. Springer, 1977.

\bibitem{MajumdarSuryFruit}
D.~Majumdar and B.~Sury.Fruit diophantine equation. {\em arXiv preprint}, 2021. Version 2.

\bibitem{McCallumPoonen2010}
W.~McCallum and B.~Poonen. The method of chabauty and coleman. {\em arXiv preprint}, 2010.

\bibitem{Mumford1970}
D.~Mumford. {\em Abelian Varieties}. Oxford University Press, 1970.

\bibitem{Neron1964}
A.~N{\'e}ron.  Mod\`eles minimaux des vari\'et\'es ab\'eliennes sur les corps locaux et globaux. {\em Publications Math{\'e}matiques de l'IH{\'E}S}, 21:5--125, 1964.

\bibitem{SharmaVaishyaFruit}
R.~Sharma and S.~Vaishya. A class of fruit diophantine equations.
 {\em Monatshefte f{\"u}r Mathematik}, 2022.

\bibitem{Silverman1986}
J.~H. Silverman, {\em The Arithmetic of Elliptic Curves}, volume 106 of {\em Graduate Texts in Mathematics}. Springer, 1986.

\bibitem{Silverman1994}
J.~H. Silverman. {\em Advanced Topics in the Arithmetic of Elliptic Curves}, volume 151 of {\em Graduate Texts in Mathematics}. Springer, 1994.

\end{thebibliography}
\end{document}